\RequirePackage{amsmath}
\documentclass{article}
\usepackage{bm}
\usepackage[T1]{fontenc}
\usepackage{enumitem}
\usepackage{multicol}
\usepackage{hyperref}

\usepackage[square,comma,numbers,sort&compress]{natbib}

\newcommand{\RT}{\mathrm{RT}}
\newcommand{\URT}{\mathrm{URT}}
\newcommand{\ART}{\mathrm{ART}}

\renewcommand{\tt}[1]{\textnormal{\texttt{#1}}}

\newcommand{\rev}[1]{{#1}^{R}}

\usepackage{color}

\usepackage{cancel}
\usepackage{tikz}
\usetikzlibrary{trees}

\usepackage{amsthm}

\newtheorem{theorem}{Theorem}
\newtheorem{Lemma}[theorem]{Lemma}

\theoremstyle{definition}
\newtheorem{Conjecture}[theorem]{Conjecture}
\newtheorem{question}[theorem]{Question}

\begin{document}
\title{The undirected repetition threshold and undirected pattern avoidance}
\author{James D.~Currie\thanks
{Supported by the Natural Sciences and Engineering Research Council of Canada (NSERC), [funding reference number 2017-03901].} and Lucas Mol\\
\\
{\normalsize The University of Winnipeg}\\
{\normalsize Winnipeg, Manitoba, Canada R3B 2E9}\\
{\normalsize \{j.currie,l.mol\}@uwinnipeg.ca}}

\date{}

\maketitle   

\begin{abstract}
\noindent
For a rational number $r$ such that $1<r\leq 2$, an \emph{undirected $r$-power} is a word of the form $xyx'$, where the word $x$ is nonempty, the word $x'$ is in $\{x,\rev{x}\}$, and we have $|xyx'|/|xy|=r$.  The \emph{undirected repetition threshold} for $k$ letters, denoted $\URT(k)$, is the infimum of the set of all $r$ such that undirected $r$-powers are avoidable on $k$ letters.  We first demonstrate that $\URT(3)=\tfrac{7}{4}$.  Then we show that $\URT(k)\geq \tfrac{k-1}{k-2}$ for all $k\geq 4$.  We conjecture that $\URT(k)=\tfrac{k-1}{k-2}$ for all $k\geq 4$, and we confirm this conjecture for $k\in\{4,5,\ldots,21\}.$  We then consider related problems in pattern avoidance; in particular, we find the undirected avoidability index of every binary pattern.  This is an extended version of a paper presented at WORDS 2019, and it contains new and improved results.

\noindent
\textbf{MSC 2010:} 68R15

\noindent
\textbf{Keywords:} Repetition thresholds, Gapped repeats, Gapped palindromes, Pattern avoidance, Patterns with reversal
\end{abstract}

\section{Introduction}\label{Intro}

A \emph{square} is a word of the form $xx$, where $x$ is a nonempty word.  An \emph{Abelian square} is a word of the form $x\tilde{x}$, where $\tilde{x}$ is an anagram (or permutation) of $x$.  
The notions of square and Abelian square can be extended to fractional powers in a natural way.  Let $1<r\leq 2$ be a rational number.  An \emph{(ordinary) $r$-power} is a word of the form $xyx$, where $x$ is a nonempty word, and $|xyx|/|xy|=r$.  An \emph{Abelian $r$-power} is a word of the form $xy\tilde{x}$, where $x$ is a nonempty word, the word $\tilde{x}$ is an anagram of $x$, and $|xy\tilde{x}|/|xy|=r$.  Here, we use the definition of Abelian $r$-power given by Cassaigne and Currie~\cite{CassaigneCurrie1999}.  We note that several distinct definitions exist (see~\cite{SamsonovShur2012,Fici2016}, for example).

In general, if $\sim$ is an equivalence relation on words that respects length (i.e., we have $|x|=|x'|$ whenever $x\sim x'$), then an \emph{$r$-power up to $\sim$} is a word of the form $xyx'$, where $x$ is a nonempty word, and we have both $x'\sim x$ and $|xyx'|/|xy|=r$. 
The notion of $r$-power up to $\sim$ generalizes ordinary $r$-powers and Abelian $r$-powers, where the equivalence relations are equality and ``is an anagram of'', respectively.

Let $\sim$ be an equivalence relation on words that respects length.  For a real number $1<\alpha\leq 2$, a word $w$ is called \emph{$\alpha$-free up to $\sim$} if no factor of $w$ is an $r$-power up to $\sim$ for $r\geq \alpha$.  The word $w$ is called \emph{$\alpha^+$-free up to $\sim$} if no factor of $w$ is an $r$-power up to $\sim$ for $r>\alpha$.  For every integer $k\geq 2$, we say that $\alpha$-powers up to $\sim$ are \emph{$k$-avoidable} if there is an infinite word on $k$ letters that is $\alpha$-free up to $\sim$, and \emph{$k$-unavoidable} otherwise.  For every integer $k\geq 2$, the \emph{repetition threshold up to $\sim$} for $k$ letters, denoted $\RT_\sim(k)$, is defined as
\[
\RT_\sim(k)=\inf\{r\colon\ \mbox{$r$-powers up to $\sim$ are $k$-avoidable}\}.
\]
Since we have only defined $r$-powers up to $\sim$ for $1<r\leq 2$, it follows that $\RT_\sim(k)\leq 2$ or $\RT_\sim(k)=\infty$ for any particular value of $k$.

It is well-known that squares are $3$-avoidable~\cite{Berstel1995}.  Thus, for $k\geq 3$, we have that $RT_{=}(k)$ is the usual \emph{repetition threshold}, denoted simply $\RT(k)$.  Dejean~\cite{Dejean1972} proved that $\RT(3)=7/4$, and conjectured that $\RT(4)=7/5$ and $\RT(k)=k/(k-1)$ for all $k\geq 5$.  This conjecture has been confirmed through the work of many authors~\cite{CurrieRampersad2011,Rao2011,Pansiot1984,Dejean1972,Carpi2007,MoulinOllagnier1992,CurrieRampersad2009Again,Noori2007}.

Let $\approx$ denote the equivalence relation ``is an anagram of''.  It is well-known that Abelian squares are $4$-avoidable~\cite{Keranen1992}.  Thus, for all $k\geq 4$, we see that $\RT_\approx(k)$ is equal to the \emph{Abelian repetition threshold} (or \emph{commutative repetition threshold}) for $k$ letters, introduced by Cassaigne and Currie~\cite{CassaigneCurrie1999}, and denoted $\ART(k)$.  Relatively less is known about the Abelian repetition threshold.  Cassaigne and Currie~\cite{CassaigneCurrie1999} give (weak) upper bounds on $\ART(k)$ in demonstrating that $\lim_{k\rightarrow\infty}\ART(k)=1$. Samsonov and Shur~\cite{SamsonovShur2012} conjecture that
$\ART(4)=9/5$ and $\ART(k)=(k-2)/(k-3)$ for all $k\geq 5$,
and give a lower bound matching this conjecture.  (Note that Samsonov and Shur define weak, semi-strong, and strong Abelian $\alpha$-powers for all real numbers $\alpha>1$.  For any rational number $r\in(1,2]$, their definitions of semi-strong Abelian $r$-power and strong Abelian $r$-power are both equivalent to our definition of Abelian $r$-power.)

For every word $x=x_1x_2\cdots x_n$, where the $x_i$ are letters, we let $\rev{x}$ denote the \emph{reversal of $x$}, defined by $\rev{x}=x_n\cdots x_2x_1$.  For example, if $x=\tt{time}$ then $\rev{x}=\tt{emit}$.  Let $\simeq$ be the equivalence relation on words defined by $x\simeq x'$ if $x'\in\{x,\rev{x}\}$.  In this article, we focus on determining $\RT_{\simeq}(k)$. We simplify our notation and terminology as follows.  We refer to $r$-powers up to $\simeq$ as \emph{undirected $r$-powers}.  These come in two types: words of the form $xyx$ are ordinary $r$-powers, while we refer to words of the from $xyx^R$ as \emph{reverse $r$-powers}.  
For example, the English words \tt{edited} and \tt{render} are undirected $\tfrac{3}{2}$-powers; \tt{edited} is an ordinary $\tfrac{3}{2}$-power, while \tt{render} is a reverse $\tfrac{3}{2}$-power.  We say that a word $w$ is \emph{undirected $\alpha$-free} if it is $\alpha$-free up to $\simeq$.  The definition of an \emph{undirected $\alpha^+$-free} word is analogous.  We let $\URT(k)=\RT_{\simeq}(k)$, and refer to this as the \emph{undirected repetition threshold} for $k$ letters.

It is clear that $\simeq$ is coarser than $=$ and finer than $\approx$.  Thus, for every rational $1<r\leq 2$, an $r$-power is an undirected $r$-power, and an undirected $r$-power is an Abelian $r$-power.  As a result, we immediately have
\[
\RT(k)\leq \URT(k)\leq \ART(k)
\]
for all $k\geq 2$.

We now describe the layout of the remainder of the article.  In Section~\ref{URT3}, we show that $\URT(3)=7/4$ using a standard morphic constuction.  In Section~\ref{Backtrack}, we demonstrate that $\URT(k)\geq (k-1)/(k-2)$ for all $k\geq 4$.  In Section~\ref{URT4}, we use a variation of the encoding introduced by Pansiot~\cite{Pansiot1984} to prove that $\URT(k)=(k-1)/(k-2)$ for $k\in\{4,5,\ldots,21\}$.  In Section~\ref{Patterns}, we consider some related problems in pattern avoidance.  In particular, we find the ``undirected avoidability index'' of every binary pattern.

In light of our results on the undirected repetition threshold, we propose the following conjecture.

\begin{Conjecture}\label{conj}
For every $k\geq 4$, we have $\URT(k)=(k-1)/(k-2)$.
\end{Conjecture}

We note that words of the form $xyx$ are sometimes referred to as \emph{gapped repeats}, and that words of the form $xy\rev{x}$ are sometimes referred to as \emph{gapped palindromes}.  In particular, an ordinary (reverse, respectively) $r$-power satisfying $r\geq 1+1/\alpha$ is called an \emph{$\alpha$-gapped repeat} (\emph{$\alpha$-gapped palindrome}, respectively).  From this perspective, the undirected repetition threshold is a measure of how large we can make the ``gaps'' of the gapped repeats and gapped palindromes in an infinite word over an alphabet of size $k$.  Algorithmic questions concerning the identification and enumeration of $\alpha$-gapped repeats and palindromes in a given word, along with some related questions, have recently received considerable attention; see~\cite{IKoppl2019, DuchonNicaudPivoteau2018, GawrychowskiManea2015, CrochemoreKolpakovKucherov2016} and the references therein.  Gapped repeats and palindromes are important in the context of DNA and RNA structures, and this has been the primary motivation for their study.

We now introduce some notation and terminology that will be used in the sequel.  For every integer $k\geq 2$, we let $\Sigma_k$ denote the alphabet $\{\tt{1},\tt{2},\dots,\tt{k}\}$.  Let $A$ and $B$ be alphabets, and let $h\colon A^*\rightarrow B^*$ be a morphism.  Using the standard notation for images of sets, we have $h(A)=\{h(a)\colon\ a\in A\},$ which we refer to as the set of \emph{blocks} of $h$.  
A set of words $P\subseteq A^*$ is called a \emph{prefix code} if no element of $P$ is a prefix of another.
If $P$ is a prefix code and $w$ is a nonempty factor of some element of $P^+$, a \emph{cut} of $w$ over $P$ is a pair $(x,y)$ such that (i) $w=xy$; and (ii) for every pair of words $p,s$ with $pws\in P^+,$ we have $px\in P^*$.  (Note that it suffices to check condition (ii) for every pair of words $p,s$ where $p$ is a prefix of a block and $s$ is a suffix of a block.)
We use vertical bars to denote cuts.  For example, over the prefix code $\{\tt{12},\tt{21}\},$ the word $\tt{11}$ has cut $\tt{1}\vert \tt{1}$.  
The prefix code that we work over will always be the set of blocks of a given morphism, and should be clear from context if it is not explicitly stated.

\section{\texorpdfstring{\boldmath{$\URT(3)=\tfrac{7}{4}$}}{URT(3)=7/4}}\label{URT3}

Dejean~\cite{Dejean1972} demonstrated that $\RT(3)=7/4$, and hence we must have $\URT(3)\geq 7/4$.  In order to show that $\URT(3)=7/4,$ it suffices to find an infinite ternary word that is undirected $\tfrac{7}{4}^+$-free.  We provide a morphic construction of such a word.  Let $f:\Sigma_3^*\rightarrow \Sigma_3^*$ be the $24$-uniform morphism defined by
\begin{align*}
    \tt{1}&\mapsto \tt{123\,132\,312\,132\,123\,213\,231\,321}\\
    \tt{2}&\mapsto \tt{231\,213\,123\,213\,231\,321\,312\,132}\\
    \tt{3}&\mapsto \tt{312\,321\,231\,321\,312\,132\,123\,213}.
\end{align*}
The morphism $f$ is similar in structure to the morphism of Dejean~\cite{Dejean1972} whose fixed point avoids ordinary $7/4^+$-powers (but not undirected $7/4^+$-powers).  Note, in particular, that $f$ is ``symmetric'' as defined by Frid~\cite{Frid2001}. 

The following theorem was also verified by one of the anonymous reviewers of the conference version of this paper using the automatic theorem proving software \tt{Walnut}~\cite{Walnut}.

\begin{theorem}\label{thm:URT3}
The word $f^\omega(\tt{1})$ is undirected $\tfrac{7}{4}^+$-free.
\end{theorem}

\begin{proof}
We first show that $f^\omega(\tt{1})$ has no factors of the form $xyx^R$ with $|x|>3|y|$ (which is equivalent to $|xy\rev{x}|/|xy|>7/4$).  By exhaustively checking all factors of length $19$ of $f^\omega(1)$, we find that $f^\omega(\tt{1})$ has no reversible factors of length greater than $18$.  So if $f^\omega(\tt{1})$ has a factor of the form $xyx^R$ with $|x|>3|y|$, then $|x|\leq 18$, and in turn $|y|<6$.  So $|xyx^R|<42$.  Every factor of length at most $41$ appears in $f^3(\tt{1})$, so by checking this prefix exhaustively we conclude that $f^\omega(\tt{1})$ has no factors of this form.

It remains to show that $f^\omega(\tt{1})$ is (ordinary) $\tfrac{7}{4}^+$-free.  Suppose towards a contradiction that $f^\omega(\tt{1})$ has a factor $xyx$ with $|x|>3|y|$.  Let $n$ be the smallest number such that a factor of this form appears in $f^n(\tt{1})$.  By exhaustive check, we have $n>3$.  First of all, if $|x|\leq 27$, then $|xyx|<63$.  Every factor of $f^\omega(\tt{1})$ of length at most $62$ appears in $f^3(\tt{1})$, so we may assume that $|x|\geq 28$.  Then $x$ contains at least one of the factors $\tt{12131}$, $\tt{23212}$, or $\tt{31323}.$  By inspection, each one of these factors determines a cut in $x$ over the blocks of $f$, say $x=s_x\vert x'\vert p_x$, where $s_x$ is a possibly empty proper suffix of a block of $f$, and $p_x$ is a possibly empty proper prefix of a block of $f$.  If $y$ is properly contained in a single block, then 
\[
xyx=s_x\vert x'\vert p_xys_x\vert x'\vert p_x.
\]
In this case, one verifies that the preimage of $xyx$ contains a square, which contradicts the minimality of $n$.  Otherwise, if $y$ is not properly contained in a single block of $f$, then $y=s_y\vert y'\vert p_y$, where $s_y$ is a possibly empty proper suffix of a block, and $p_y$ is a possibly empty proper prefix of a block.  Then
\[
xyx=s_x\vert x'\vert p_xs_y\vert y'\vert p_ys_x\vert x'\vert p_x,
\]
which appears internally as
\[
\vert p_ys_x\vert x'\vert p_xs_y\vert y'\vert p_ys_x\vert x'\vert p_xs_y\vert.
\]
The preimage of this factor is $ax_1by_1ax_1b,$ where $f(a)=p_ys_x$, $f(b)=p_xs_y$, $f(x_1)=x'$, and $f(y_1)=y'.$  Then $19|ax_1b|\geq |x|>3|y|\geq 3\cdot 19|y_1|,$ or equivalently $|ax_1b|>3|y_1|,$ which contradicts the minimality of $n$.
\end{proof}

\noindent
Thus, we conclude that $\URT(3)=\RT(3)=\frac{7}{4}$.  We will see in the next section that $\URT(k)$ is strictly greater than $\RT(k)$ for every $k\geq 4$.  

\section{A lower bound on \texorpdfstring{\boldmath{$\URT(k)$}}{URT(k)} for \texorpdfstring{\boldmath{$k\geq 4$}}{k>=4}}\label{Backtrack}

Here, we prove that $\URT(k)\geq (k-1)/(k-2)$ for $k\geq 4$.

\begin{theorem}\label{LowerBoundBacktrack}
If $k\geq 4$, then $\URT(k)\geq \tfrac{k-1}{k-2}$, and the longest word over $\Sigma_k$ that is undirected $(k-1)/(k-2)$-free has length $k+3$.
\end{theorem}

\begin{proof}
For $k\in\{4,5\}$, the statement is checked by a standard backtracking algorithm, which we performed both by hand and by computer.  We now provide a general backtracking argument for all $k\geq 6$.

Fix $k\ge 6$, and suppose that $w\in \Sigma_k^*$ is a word of length $k+4$ that is undirected $(k-1)/(k-2)$-free. It follows that at least $k-2$ letters must appear between any two repeated occurrences of the same letter in $w$, so that any length $k-1$ factor of $w$ must contain $k-1$ distinct letters.  So we may assume that $w$ has prefix $\tt{12}\cdots\tt{(k-1)}$.  Further, given any prefix $u$ of $w$ of length at least $k-1$, there are only two possibilities for the next letter in $w$, as it must be distinct from the $k-2$ distinct letters preceding it.  These possibilities are enumerated in the tree of Figure~\ref{BacktrackingTree}.

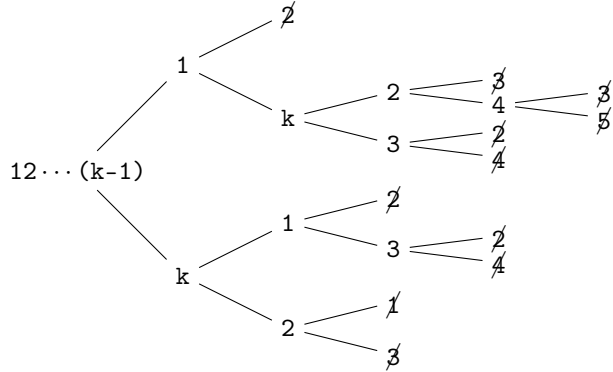
\begin{figure}[htb]
    \centering
\tikzstyle{level 1}=[level distance=2cm, sibling distance=4cm]
\tikzstyle{level 2}=[level distance=2cm, sibling distance=2cm]
\tikzstyle{level 3}=[level distance=2cm, sibling distance=1cm]
\tikzstyle{level 4}=[level distance=2cm, sibling distance=0.5cm]
\tikzstyle{level 5}=[level distance=2cm, sibling distance=0.5cm]

\tikzstyle{start} = []
\tikzstyle{letter} = []

\begin{tikzpicture}[grow=right, sloped,scale=0.7]
\node[start,left] {$\tt{12}\cdots\tt{(k-1)}$}
    child {
        node[letter] {\tt{k}}        
            child {
                node[letter] {\tt{2}}
                child {
                    node[letter] {\cancel{\tt{3}}}
                    edge from parent
                }
                child {
                    node[letter] {\cancel{\tt{1}}}
                    edge from parent
                }
                edge from parent
            }
            child {
                node[letter] {\tt{1}}
                child {
                    node[letter] {\tt{3}}
                    child {
                        node[letter] {\cancel{\tt{4}}}
                        edge from parent
                    }
                    child {
                        node[letter] {\cancel{\tt{2}}}
                        edge from parent
                    }
                    edge from parent
                }
                child {
                    node[letter] {\cancel{\tt{2}}}
                    edge from parent
                }
                edge from parent
            }
            edge from parent 
    }
    child {
        node[letter] {\tt{1}}        
        child {
                node[letter] {\tt{k}}
                child {
                    node[letter] {\tt{3}}
                    child {
                        node[letter] {\cancel{\tt{4}}}
                        edge from parent
                    }
                    child {
                        node[letter] {\cancel{\tt{2}}}
                        edge from parent
                    }
                    edge from parent
                }
                child {
                    node[letter] {\tt{2}}
                    child {
                        node[letter] {\tt{4}}
                        child {
                            node[letter] {\cancel{\tt{5}}}
                            edge from parent
                        }
                        child {
                            node[letter] {\cancel{\tt{3}}}
                            edge from parent
                        }
                        edge from parent
                    }
                    child {
                        node[letter] {\cancel{\tt{3}}}
                        edge from parent
                    }
                    edge from parent
                }
                edge from parent
            }
            child {
                node[letter] {\cancel{\tt{2}}}
                edge from parent
            }
        edge from parent         
    };
\end{tikzpicture}
    \caption{The tree of undirected $(k-1)/(k-2)$-power free words on $k$ letters.}
    \label{BacktrackingTree}
\end{figure}

We now explain why each word corresponding to a leaf of the tree contains an undirected $r$-power for some $r\geq (k-1)/(k-2)$.  We examine the leaves from top to bottom, and use the fact that $(k+1)/(k-1)>(k+2)/k>(k-1)/(k-2)$ when $k\geq 6$.
\begin{itemize}
\item The factor $\tt{12}\cdots\tt{(k-1)12}$ is an ordinary $(k+1)/(k-1)$-power.  
\item The factor $\tt{23}\cdots\tt{(k-1)1k23}$ is an ordinary  $(k+2)/k$-power.  
\item The factor $\tt{34}\cdots\tt{(k-1)1k243}$ is a reverse $(k+2)/k$-power.
\item The factor $\tt{45}\cdots\tt{(k-1)1k245}$ is an ordinary  $(k+1)/(k-1)$-power.
\item The factor $\tt{23}\cdots\tt{(k-1)1k32}$ is an ordinary $(k+2)/k$-power.
\item The factor $\tt{34}\cdots\tt{(k-1)1k34}$ is an ordinary $(k+1)/(k-1)$-power.
\item The factor $\tt{12}\cdots\tt{(k-1)k12}$ is an ordinary  $(k+2)/k$-power.
\item The factor $\tt{23}\cdots\tt{(k-1)k132}$ is a reverse $(k+2)/k$-power.
\item The factor $\tt{34}\cdots\tt{(k-1)k134}$ is an ordinary $(k+1)/(k-1)$-power.
\item The factor $\tt{12}\cdots\tt{(k-1)k21}$ is a reverse $(k+2)/k$-power.
\item The factor $\tt{23}\cdots\tt{(k-1)k23}$ is an ordinary $(k+1)/(k-1)$-power.\qedhere
\end{itemize}
\end{proof}

Conjecture~\ref{conj} proposes that the value of $\URT(k)$ matches the lower bound of Theorem~\ref{LowerBoundBacktrack} for all $k\geq 4$.  In the next section, we confirm Conjecture~\ref{conj} for some small values of $k$.

\section{\texorpdfstring{\boldmath{$\URT(k)=\tfrac{k-1}{k-2}$}}{URT(k)=(k-1)/(k-2)} for all \texorpdfstring{\boldmath{$k\in\{4,5,\ldots,21\}$}}{k in {4,5,...,21}}}\label{URT4}

First we explain why we rely on a different type of construction than the one we used to prove that $\URT(3)=\tfrac{7}{4}$ in Section~\ref{URT3}.  A morphism $h:A^*\rightarrow B^*$ is called $\alpha$-free ($\alpha^+$-free, respectively) if it maps every $\alpha$-free ($\alpha^+$-free, respectively) word in $A^*$ to an $\alpha$-free ($\alpha^+$-free, respectively) word in $B^*$.  The morphism $h$ is called \emph{growing} if $h(a)>1$ for all $a\in A^*$.  Brandenburg~\cite{Brandenburg1983} demonstrated that for every $k\geq 4$, there is no growing $\RT(k)^+$-free morphism from $\Sigma_k^*$ to $\Sigma_k^*$.  By a minor modification of his proof, one can show that there is no growing $(k-1)/(k-2)^+$-free morphism from $\Sigma_k^*$ to $\Sigma_k^*$.  While this does not entirely rule out the possibility that there is a morphism from $\Sigma_k^*$ to $\Sigma_k^*$ whose fixed point is $(k-1)/(k-2)^+$-free, it suggests that a different type of construction may be required.  Our constructions rely on a variation of the encoding introduced by Pansiot~\cite{Pansiot1984} in showing that $\RT(4)=7/5$.  Pansiot's encoding was later used in all subsequent work on Dejean's Conjecture.

\subsection{A ternary encoding}

We first describe an alternate definition of ordinary $r$-powers which will be useful in this section.  A word $w=w_1\cdots w_n$, where the $w_i$ are letters, is \emph{periodic} if for some positive integer $q$, we have $w_{i+q}=w_i$ for all $1\leq i\leq n-q$.  In this case, the integer $q$ is called a \emph{period} of $w$.  The \emph{exponent} of $w$, denoted $\exp(w),$ is the ratio between its length and its minimal period.  If $r=\exp(w)$, then $w$ is an \emph{$r$-power}.\footnote{If $r\leq 2$, then $w$ is an $r$-power as we have defined it in Section~\ref{Intro}. If $r>2$, then we take this as the definition of an (ordinary) $r$-power. For example, the English word \tt{alfalfa} has minimal period $3$ and exponent $\tfrac{7}{3},$ so it is a $\tfrac{7}{3}$-power.}  We can write any $r$-power $w$ as $w=pe$, where $|pe|/|p|=r$ and $e$ is a prefix of $pe$.  In this case, we say that $e$ is the \emph{excess} of the $r$-power $w$.

Suppose that $w\in \Sigma_k^*$ is an undirected $(k-1)/(k-2)^+$-free word that contains at least $k-1$ distinct letters.  Write $w=w_1w_2\cdots w_n$ with $w_i\in \Sigma_k$.  Certainly, every length $k-2$ factor of $w$ contains $k-2$ distinct letters, and it is easily checked that every length $k$ factor of $w$ contains at least $k-1$ distinct letters.

Now let $w\in \Sigma_k^*$ be any word containing at least $k-1$ distinct letters and satisfying these two properties:
\begin{itemize}
    \item Every length $k-2$ factor of $w$ contains $k-2$ distinct letters; and
    \item Every length $k$ factor of $w$ contains at least $k-1$ distinct letters.
\end{itemize}
Let $u$ be the shortest prefix of $w$ containing $k-1$ distinct letters.  We see immediately that $u$ has length $k-1$ or $k$.  Write $w=uv$, where $v=v_1v_2\cdots v_{n}$ with $v_i\in \Sigma_k$.  Define $p_0=u$ and $p_i=uv_1\cdots v_i$ for all $i\in\{1,\dots,n\}$.  For all $i\in\{0,1,\dots,n\}$, the prefix $p_i$ determines a permutation
\[
r_i=\begin{pmatrix}
\tt{1} & \tt{2} & \dots & \tt{k}\\ 
r_i[\tt{1}] & r_i[\tt{2}] & \dots & r_i[\tt{k}]
\end{pmatrix},
\]
of the letters of $\Sigma_k$, which ranks the letters of $\Sigma_k$ by the index of their final appearance in $p_i$.  In other words, the word $r_i[\tt{3}] \cdots r_i[\tt{k}]$ is the length $k-2$ suffix of $p_i$, and of the two letters in $\Sigma_k\backslash\{r_i[3],\dots,r_i[k]\}$, the letter $r_i[\tt{2}]$ is the one that appears last in $p_i$. Note that the final letter $r_i[\tt{1}]$ may not even appear in $p_i$.  For example, on $\Sigma_6,$ the prefix \tt{123416} gives rise to the permutation
\[
\begin{pmatrix}
\tt{1} & \tt{2} & \tt{3} & \tt{4} & \tt{5} & \tt{6}\\
\tt{5} & \tt{2} & \tt{3} & \tt{4} & \tt{1} & \tt{6}
\end{pmatrix}.
\]
Since every factor of length $k-2$ in $w$ contains $k-2$ distinct letters, for any $i\in\{1,\dots,n\}$, the letter $v_i$ must belong to the set $\{r_{i-1}[1],r_{i-1}[2],r_{i-1}[3]\}.$  This allows us to encode the word $w$ over a ternary alphabet, as described explicitly below.

For $1\leq i\leq n$, define $t(w)=t_1\cdots t_{n}$, where for all $1\leq i\leq n$, we have
\[
t_i=\begin{cases}
\tt{1}, & \text{if $v_i=r_{i-1}[1]$};\\
\tt{2}, & \text{if $v_i=r_{i-1}[2]$};\\
\tt{3}, & \text{if $v_i=r_{i-1}[3]$}.
\end{cases}
\]

\noindent
For example, on $\Sigma_5$, for the word $w=\tt{12342541243},$
the shortest prefix containing $4$ distinct letters is $1234$, and $w$ has encoding $t(w)=\tt{3131231}.$
Given the shortest prefix of $w$ containing $k-1$ distinct letters, and the encoding $t(w)$, we can recover $w$.  Moreover, if $w$ has period $q<n$, then so does $t(w)$.  The exponent $|w|/q$ of $w$ corresponds to an exponent $|v|/q$ of $t(w)$.

Let $S_k$ denote the symmetric group on $\Sigma_k$ with left multiplication.  Define a morphism $\sigma:\Sigma_3^*\rightarrow S_k$ by
\begin{align*}
\sigma(1)&=\begin{pmatrix}
    1 & 2 & 3 & 4 & \dots & k-1 & k \\
    2 & 3 & 4 & 5 & \dots & k & 1
  \end{pmatrix}\\
\sigma(2)&=\begin{pmatrix}
    1 & 2 & 3 & 4 & \dots & k-1 & k \\
    1 & 3 & 4 & 5 & \dots & k & 2
  \end{pmatrix}\\
\sigma(3)&=\begin{pmatrix}
    1 & 2 & 3 & 4 & \dots & k-1 & k \\
    1 & 2 & 4 & 5 & \dots & k & 3
  \end{pmatrix}.
\end{align*}
One proves by induction that $r_0\sigma(t(p_i))=r_i$.  It follows that if $w=pe$ has period $|p|$, and $e$ contains at least $k-1$ distinct letters, then the length $|p|$ prefix of $t(w)$ lies in the kernel of $\sigma$.  In this case, the word $t(w)$ is called a \emph{kernel repetition}. For example, over $\Sigma_4$, the word
\[
w=\tt{123243414212324}
\]
has period $10$, and excess $\tt{12324}$.  Hence, the encoding $t(w)=\tt{312313123131}$ is a kernel repetition; one verifies that $\sigma(\tt{3123131231})=\mbox{id}$.

The following straightforward lemma will be used to bound the length of reversible factors in the words that we construct.

\begin{Lemma}\label{ReversibleFactors}
Let $k\geq 4$, and let $w\in \Sigma_k^*$ be a word with encoding $t(w)\in\Sigma_3^*$.  Suppose that neither $\tt{312}$ nor $\tt{322}$ is a factor of $t(w)$.  Let $u$ be a factor of $w$ whose encoding $t(u)$ contains the factor $\tt{1231}$.  Then $\rev{u}$ is not a factor of $w$.
\end{Lemma}

\begin{proof}
Since $t(u)$ contains the factor $\tt{1231}$, the word $u$ contains some permutation of the factor $\tt{123}\cdots \tt{(k-1)13k}$.  By inspection, the reversal of this word, namely $\tt{k31(k-1)(k-2)}\cdots\tt{321}$, has encoding  $\tt{312}$ or $\tt{322}$, neither of which is a factor of the encoding $t(w)$ by assumption.  We conclude that $\rev{u}$ is not a factor of $w$.
\end{proof}


\newpage

\subsection{Constructions}

For $k\in\{4,5,\ldots,21\}$, define the morphism $f_k:\Sigma_2^*\rightarrow \Sigma_2^*$ as follows:
\begin{multicols}{2}%
\footnotesize%
\noindent
\begin{align*}
f_4(\tt{1})&=\tt{12111211212}\\
f_4(\tt{2})&=\tt{11121211211}\\[5pt]
f_5(\tt{1})&=	\tt{121121212}\\
f_5(\tt{2})&=\tt{121121211}\\[5pt]
f_6(\tt{1})&=\tt{1122212}\\
f_6(\tt{2})&=\tt{1122211}\\[5pt]
f_7(\tt{1})&=	\tt{121212112121211}\\
f_7(\tt{2})&=\tt{211211112121112}\\[5pt]
f_8(\tt{1})&=\tt{11212211112122}\\
f_8(\tt{2})&=\tt{12211212121221}\\[5pt]
f_9(\tt{1})&=\tt{11212121121121121}\\
f_9(\tt{2})&=\tt{21212112121112112}\\[5pt]
f_{10}(	\tt{1})&=\tt{1211121212121211121112111}\\
f_{10}(	\tt{2})&=\tt{2111121121121112111121112}\\[5pt]
f_{11}(\tt{1})&=\tt{1121121121121211211}\\
f_{11}(\tt{2})&=\tt{2121121121121211112}\\[5pt]
f_{12}(	\tt{1})&=\tt{11211112111121121121211}\\
f_{12}(	\tt{2})&=\tt{21211212112111121121212}
\end{align*}%
\begin{align*}
f_{13}(	\tt{1})&=\tt{1211212111212121}\\
f_{13}(	\tt{2})&=\tt{2121212111212112}\\[5pt]
f_{14}(	\tt{1})&=\tt{111211212111211121121}\\
f_{14}(	\tt{2})&=\tt{211211112111212121112}\\[5pt]
f_{15}(	\tt{1})&=\tt{12121211112112121}\\
f_{15}(	\tt{2})&=\tt{21121212112112112}\\[5pt]
f_{16}(	\tt{1})&=\tt{1121112111121112112112111211211}\\
f_{16}(\tt{2})&=\tt{2121112112121112112112111211212}\\[5pt]
f_{17}(	\tt{1})&=\tt{1211121121121211211}\\
f_{17}(	\tt{2})&=\tt{2111121121121211212}\\[5pt]
f_{18}(	\tt{1})&=\tt{112112112111211211211}\\
f_{18}(	\tt{2})&=\tt{211112112121211211112}\\[5pt]
f_{19}(	\tt{1})&=\tt{121112112111212111121}\\
f_{19}(	\tt{2})&=\tt{211112112111212111212}\\[5pt]
f_{20}(\tt{1})&=\tt{112121211212111211211112121211}\\
f_{20}(\tt{2})&=\tt{212121211112111211212112121112}\\[5pt]
f_{21}(	\tt{1})&=\tt{12111211211112121212121}\\
f_{21}(	\tt{2})&=\tt{21111211211212121212112}
\end{align*}  
\end{multicols}
\noindent
For all $k\not\in\{5,6,8\}$, define $g_k:\Sigma_2^*\rightarrow \Sigma_3^*$ by
{ \footnotesize
\begin{align*}
g_k(\tt{1})&=\tt{31}\\
g_k(\tt{2})&=\tt{12}.
\end{align*}}%
For $k\in\{5,6,8\}$, define $g_k:\Sigma_2^*\rightarrow \Sigma_3^*$ as follows:
\begin{multicols}{3}
\footnotesize
\noindent
\begin{align*}
g_5(\tt{1})&=\tt{3111}\\
g_5(\tt{2})&=\tt{3112}
\end{align*}
\begin{align*}
g_6(\tt{1})&=\tt{31112}\\
g_6(\tt{2})&=\tt{31131}
\end{align*}
\begin{align*}
g_8(\tt{1})&=\tt{11231}\\
g_8(\tt{2})&=\tt{11313}
\end{align*}
\end{multicols}


\begin{theorem}\label{Main}
Fix $k\in\{4,5,\ldots,21\}$.  Let $\bm{w}_k$ be the word over $\Sigma_k$ with prefix $\tt{12}\cdots\tt{(k-1)}$ and encoding $g_k(f_k^\omega(\tt{1}))$.  Then $\bm{w}_k$ is undirected $(k-1)/(k-2)^+$-free.
\end{theorem}

The remainder of this section is devoted to proving Theorem~\ref{Main}.  Essentially, we adapt and extend the technique first used by Moulin-Ollagnier~\cite{MoulinOllagnier1992}.  A simplified version of Moulin-Ollagnier's technique, which we follow fairly closely, is exhibited by Currie and Rampersad~\cite{CurrieRampersad2011}.  

For the remainder of this section, we use notation as in Theorem~\ref{Main}, but we omit the subscripts on $\bm{w}$, $f$, and $g$ for convenience.  We let $r=|f(\tt{1})|$ and $r_g=|g(\tt{1})|$, i.e., we say that $f$ is $r$-uniform, and $g$ is $r_g$-uniform.   We use the following properties of $f$ and $g$ several times:
\begin{itemize}
\item Every factor of $f^\omega(\tt{1})$ of length $r$ contains a cut over the blocks of $f$, and every factor of $g(f^\omega(\tt{1}))$ of length $r_g$ contains a cut over the blocks of $g$.  \item The blocks $f(\tt{1})$ and $f(\tt{2})$ end in different letters, and the blocks $g(\tt{1})$ and $g(\tt{2})$ end in different letters.
\end{itemize}
The first property was verified by computer.

Before proceeding with the proof of Theorem~\ref{Main}, we discuss the kernel repetitions that appear in $g(f^\omega(\tt{1}))$.
Let the factor $v=pe$ of $g(f^\omega(\tt{1}))$ be a kernel repetition with period $q$; say $g(f^\omega(\tt{1}))=xv\bm{y}$.  Let $V=x'vy'$ be the maximal period $q$ extension of the occurrence $xv\bm{y}$ of $v$.  Write $x=Xx'$ and $\bm{y}=y'\bm{Y},$ so that $g(f^\omega(\tt{1}))=XV\bm{Y}$.  Write $V=PE=EP',$ where $|P|=q$.  By the periodicity of $PE$, the factor $P$ is conjugate to $p$, and hence $P$ is in the kernel of $\sigma$.  Suppose that $E$ contains a cut over the blocks of $g$.  Then we may write $E$ uniquely in the form $E=\eta''g(\eta)\eta'$, where the word $\eta\in\Sigma_2^*$, the word $\eta''$ is a proper suffix of $g(\tt{1})$ or $g(\tt{2})$, and the word $\eta'$ is a proper prefix of $g(\tt{1})$ or $g(\tt{2})$.   Similarly, we may write $PE=\pi''g(\pi\eta)\eta'$, where the word $\pi\in\Sigma_2^*$ and the word $\pi''$ is a proper suffix of $g(\tt{1})$ or $g(\tt{2})$.  Since $E$ is a prefix of $PE$, and since $g(\tt{1})$ and $g(\tt{2})$ end in different letters, it follows from the maximality of $V$ that $\pi''=\eta''=\varepsilon$.  Finally, by the maximality of $V$, we have that $\eta'=\chi_g$, the longest common prefix of $g(\tt{1})$ and $g(\tt{2})$.  So we have $P=g(\pi)$ and $E=g(\eta)\chi_g$.   Since $E$ is a prefix of $PE$, we have that $\eta$ is a prefix of $\pi\eta$.  We see that $|P|=r_g|\pi|$ and $|E|= r_g|\eta|+|\chi_g|$.

Let $\tau:\Sigma_2^*\rightarrow S_k$ be the composite morphism $\sigma\circ g$.  
Since $P$ is in the kernel of $\sigma$, we see that
\[
\tau(\pi)=\sigma(g(\pi))=\sigma(P)=\mathrm{id},
\]
i.e., the word $\pi$ is in the kernel of $\tau$.

Now set $\pi_0=\pi$ and $\eta_0=\eta$.  By the maximality of $PE,$ the repetition $\pi\eta=\pi_0\eta_0$ must be a maximal repetition with period $|\pi_0|$ (i.e., it cannot be extended).  If $\eta_0$ has a cut, then it follows by arguments similar to those used above that $\pi_0=f(\pi_1)$ and $\eta_0=f(\eta_1)\chi_f$, where $\eta_1$ is a prefix of $\pi_1$ and $\chi_f$ is the longest common prefix of $f(\tt{1})$ and $f(\tt{2})$.  One checks that there is an element $\phi\in S_k$ such that
\begin{align*}
\phi\cdot \tau(f(\tt{a}))\cdot \phi^{-1}=\tau(\tt{a})
\end{align*}
for every $\tt{a}\in\{\tt{1},\tt{2}\}$, i.e., the morphism $\tau$ satisfies the ``algebraic property'' described by Moulin-Ollagnier~\cite{MoulinOllagnier1992}.  It follows that $\pi_1$ is in the kernel of $\tau$.  We can repeat this process until we reach a repetition $\pi_s\eta_s$ whose excess $\eta_s$ has no cut.  Recalling that $f$ is an $r$-uniform morphism, we have
\[
|\pi_0|=r^s|\pi_s|
\]
and 
\begin{align*}
|\eta_0|&=r^s|\eta_s|+|\chi_f|\sum_{i=0}^{s-1}r^i=r^s|\eta_s|+|\chi_f|\cdot \frac{r^s-1}{r-1}.
\end{align*}

\begin{proof}[Proof of Theorem~\ref{Main}]
Let $u$ be the word with prefix $\tt{12}\cdots \tt{(k-1)}$ and encoding $t(u)=g(f^3(\tt{1}))$.  Note that $t(u)$ is a prefix of $t(\bm{w})=g(f^\omega(\tt{1}))$, and hence $u$ is a prefix of $\bm{w}$.  We begin by verifying computationally that $u$ is undirected $(k-1)/(k-2)^+$-free.  This fact will be used several times in the proof.

We first show that $\bm{w}$ contains no reverse $r$-power with $r>(k-1)/(k-2)$.  We verify computationally that there is a finite number $N$ such that every factor of $g(f^\omega(\tt{1}))$ of length $N$ contains the factor $\tt{1231}$.  (While the exact value of $N$ depends on $k$, we have $N\leq 23$ for every $k\in\{4,5,\ldots,21\}$.)  Further, since $\tt{312}$ and $\tt{322}$ are not factors of $g(f^\omega(\tt{1}))$, we conclude by Lemma~\ref{ReversibleFactors} that no factor of $\bm{w}$ of length $N+k$ is reversible.  Thus, if $xy\rev{x}$ is a factor of $\bm{w}$ with $|xy\rev{x}|/|xy|>(k-1)/(k-2)$, then $|x|\leq N+k-1$.  In turn, we have $|xy\rev{x}|< (k-1)(N+k-1)$.  We verify computationally that every factor of $t(\bm{w})$ of length less than $(k-1)(N+k-1)$ appears in $t(u)$, and hence some permutation of every factor of $\bm{w}$ appears in $u$.  Since $u$ is undirected $(k-1)/(k-2)^+$-free, we conclude that $\bm{w}$ contains no reverse $r$-power with $r>(k-1)/(k-2)$.

It remains to show that $\bm{w}$ is ordinary $(k-1)/(k-2)^+$-free.  Suppose to the contrary that $v=pe$ is a factor of $\bm{w}$ such that $e$ is a prefix of $pe$ and $|pe|/|p|>(k-1)/(k-2)$.  We may assume that $pe$ is maximal with respect to having period $|p|$.  If $e$ has less than $k-1$ distinct letters, then $|e|\leq k-1$.  In turn, we have $|pe|<(k-1)^2$.  since every factor of $t(\bm{w})$ of length less than $(k-1)^2$ occurs in $t(u)$, and $u$ is undirected $(k-1)/(k-2)^+$-free, we may assume that $e$ has at least $k-1$ distinct letters.  In this case, let $V=t(pe)$, and let $P$ be the length $|p|$ prefix of $V$.  So $V=PE$, where $E$ is a prefix of $P$, and $P$ is in the kernel of $\sigma$.  Evidently, we have $|e|\leq |E|+k$.  If $E$ does not contain a cut, then $|E|< r_g$, and hence $|e|< r_g+k$.  It follows that $|pe|< (k-1)(r_g+k-1)$.  Since every factor of $t(\bm{w})$ of length less than $(k-1)(r_g+k-1)$ occurs in $t(u)$, and $u$ is undirected $(k-1)/(k-2)^+$-free, we may assume that $E$ contains a cut.

Since $E$ contains a cut, by the discussion immediately preceding this proof, we can find a factor $\pi_s\eta_s$ of $f^\omega(\tt{1})$ such that $\eta_s$ is a prefix of $\pi_s\eta_s$, the word $\pi_s$ is in the kernel of $\tau$, and $\eta_s$ does not contain a cut.  Then we have
\begin{align*}
    \frac{1}{k-2}&< \frac{|e|}{|p|}\\
    &\leq \frac{|E|+k}{|P|}\\
    &\leq \frac{r_g|\eta_0|+|\chi_g|+k}{r_g|\pi_0|}\\
    &=\frac{r_g\left(r^s|\eta_s|+|\chi_f|\cdot \frac{r^s-1}{r-1}\right)+|\chi_g|+k}{r_gr^s|\pi_s|}\\
    &=\frac{1}{|\pi_s|}\left[|\eta_s|+\frac{|\chi_f|}{r-1}\cdot \frac{r^s-1}{r^s}+\frac{|\chi_g|+k}{r_gr^s}\right]\\
    &\leq\frac{1}{|\pi_s|}\left[|\eta_s|+\frac{|\chi_f|}{r-1}+\frac{|\chi_g|+k}{r_gr^s}\right]
\end{align*}
Thus, we have
\begin{align}\label{General}
|\pi_s|<(k-2)\left[|\eta_s|+\frac{|\chi_f|}{r-1}+\frac{|\chi_g|+k}{r_g}\right].
\end{align}
Note also that if $s\geq 1$, then we have
\begin{align}\label{NotZero}
|\pi_s|<(k-2)\left[|\eta_s|+\frac{|\chi_f|}{r-1}+\frac{|\chi_g|+k}{r_gr}\right].
\end{align}

Since every factor of length $r$ in $f^\omega(\tt{1})$ contains a cut, we must have $|\eta_s|\leq r-1$.  Putting this together with~(\ref{General}), we find that
\[
|\pi_s\eta_s|\leq (k-2)\left[r-1+\frac{|\chi_f|}{r-1}+\frac{|\chi_g|+k}{r_gr}\right]+r-1.
\]
This is a constant bound on $|\pi_s\eta_s|$, which enables us to verify by computer that $\pi_s\eta_s$ must appear in $f^3(\tt{1})$ if $k\neq 18$, and in $f^4(\tt{1})$ if $k=18$.  So we complete a search of $f^3(\tt{1})$ (or $f^4(\tt{1})$ if $k=18$) to find all possible words $\pi_s\eta_s$ such that $\eta_s$ is a prefix of $\pi_s$,  the word $\pi_s$ is in the kernel of $\tau$, and the inequalities $|\eta_s|\leq r-1$ and (\ref{General}) are satisfied. 

First of all, if $k\geq 6$, then we find that no such word $\pi_s\eta_s$ exists, and we conclude immediately that $\bm{w}$ is undirected $(k-1)/(k-2)^+$-free.

If $k=4$, then we find only the following possibilities:
\begin{itemize}
\item $\pi_s=\tt{111}$ and $\eta_s=\varepsilon$;
\item $\pi_s=\tt{112112}$ and $\eta_s=\tt{1}$;
\item $\pi_s=\tt{121121}$ and $\eta_s=\tt{1}$.
\end{itemize}
For each of these possibilities, we see that the stricter inequality~(\ref{NotZero}) is not satisfied, and hence we must have $s=0$.  But this is impossible, because then $pe$ is in the prefix $u$ of $\bm{w}$ encoded by $g(f^3(\tt{1}))$, which we have already verified to be undirected $(k-1)/(k-2)^+$-free.  (Note that one can also recover $pe$ from $\pi_s\eta_s=\pi_0\eta_0$ in each case and verify by inspection that $pe$ has exponent at most $(k-1)/(k-2)$; these candidates for $\pi_s\eta_s$ arise from our search only because of slack in some of the inequalities leading to~(\ref{General}).)

If $k=5$, then we find only one possibility, namely $\pi_s=\tt{1212112121}$ and $\eta_s=\tt{1}$.  Again, we see that~(\ref{NotZero}) is not satisfied, and hence we must have $s=0$.  But as for the case $k=4$, this is impossible.

Therefore, we conclude in all cases that $\bm{w}$ is undirected $(k-1)/(k-2)^+$-free.  
\end{proof}

\section{Undirected pattern avoidance}\label{Patterns}


Let $V$ be an alphabet of letters called \emph{variables}, and let $p=p_1p_2\cdots p_n$ be a word with $p_i\in V$.  In this context, the word $p$ is called a \emph{pattern}.  If $\sim$ is an equivalence relation on words, then we say that the word $w$ \emph{encounters $p$ up to $\sim$} if $w$ contains a factor of the form $X_1X_2\cdots X_n$, where each word $X_i$ is nonempty and $X_i\sim X_j$ whenever $p_i=p_j$.  The word $X_1X_2\cdots X_n$ is called an \emph{instance} of $p$ up to $\sim$.  If $w$ contains no instance of $p$ up to $\sim$, then we say that $w$ \emph{avoids $p$ up to $\sim$}.  A pattern $p$ is \emph{$k$-avoidable up to $\sim$} if there is an infinite word on a $k$-letter alphabet that avoids $p$ up to $\sim$.  Otherwise, the pattern $p$ is \emph{$k$-unavoidable up to $\sim$}.  Finally, the pattern $p$ is \emph{avoidable up to $\sim$} if it is $k$-avoidable for some $k$, and \emph{unavoidable up to $\sim$} otherwise.

When $\sim$ is equality, we recover the ordinary notion of \emph{pattern avoidance} (see~\cite{CassaigneChapter,Gamard2018,OchemRosenfeld2020}, for example).  When $\sim$ is $\approx$ (i.e., ``is an anagram of''), we recover the notion of \emph{Abelian pattern avoidance} (see~\cite{CurrieLinek2001,CurrieVisentin2008,Rosenfeld2016}, for example).  In this section, we consider pattern avoidance up to $\simeq$, or \emph{undirected pattern avoidance}.  

While there are patterns that are avoidable in the ordinary sense but not in the Abelian sense~\cite[Lemma 3]{CurrieLinek2001}, every pattern that is avoidable in the ordinary sense must also be avoidable up to $\simeq$, as we show below.  For words $u=u_0u_1\cdots$ and $v=v_0v_1\cdots$ of the same length (possibly infinite) over alphabets $A$ and $B$, respectively, the \textit{direct product} of $u$ and $v,$ denoted $u\otimes v,$ is the word on alphabet $A\times B$ defined by
\[
u\otimes v=(u_0,v_0)(u_1,v_1)\cdots.
\]

\begin{theorem}\label{Avoidable}
Let $p$ be a pattern.  Then $p$ is avoidable in the ordinary sense if and only if $p$ is avoidable up to $\simeq$.
\end{theorem}

\begin{proof}
If $p$ is unavoidable in the ordinary sense, then clearly $p$ is unavoidable up to $\simeq.$  So suppose that $p$ is avoidable in the ordinary sense, and let $\bm{u}$ be an infinite word avoiding $p$.  We claim that the direct product $\bm{u}\otimes (\tt{123})^\omega$ avoids $p$ up to $\simeq$.  Write $p=p_1p_2\cdots p_n$, where the $p_i$ are variables.  Suppose towards a contradiction that $\bm{u}\otimes(\tt{123})^\omega$ contains an instance $P_1P_2\cdots P_n$ of $p$ up to $\simeq$.  Then we have $P_i\simeq P_j$ whenever $p_i=p_j$.  Since the only nonempty reversible factors of $(123)^\omega$ have length $1$, we must in fact have $P_i=P_j$ whenever $p_i=p_j$.  But that means that $\bm{u}\otimes(\tt{123})^\omega$, and hence $\bm{u}$, contains an instance of $p$ in the ordinary sense, which is a contradiction. 
\end{proof}

Questions concerning the $k$-avoidability of patterns up to $\simeq$ appear to be more interesting.  The \emph{avoidability index} of a pattern $p$ up to $\sim$, denoted $\lambda_\sim(p)$, is the least positive integer $k$ such that $p$ is $k$-avoidable up to $\sim$, or $\infty$ if $p$ is unavoidable.  In general, for any pattern $p$, we have
\[
\lambda_{=}(p)\leq \lambda_{\simeq}(p)\leq \lambda_{\approx}(p).
\]
The construction of Theorem~\ref{Avoidable} illustrates that $\lambda_{\simeq}(p)\leq 3\lambda_{=}(p)$, though we suspect that this bound is not tight for any avoidable pattern $p$.

In the remainder of this section, we determine the undirected avoidability index of every pattern on at most two variables.  We begin by finding the avoidability index of the unary patterns up to $\simeq$ using known results.

\begin{theorem}\label{UnaryIndex}
$\lambda_\simeq(x^k)=\begin{cases}
3, &\text{if $k\in\{2,3\}$};\\
2, &\text{if $k\geq 4$}.
\end{cases}$
\end{theorem}

\begin{proof}
The fact that $\lambda_{\simeq}(xx)\leq 3$ follows from Theorem~\ref{thm:URT3}.  Alternatively, note that for every nonempty word $x$, the word $x\rev{x}$ contains a length $2$ square at its center, so in fact every square-free word avoids $xx$ in the undirected sense.  So $\lambda_{\simeq}(xx)=\lambda_{=}(xx)=3$.  

Backtracking by computer, one finds that the longest binary word avoiding $xxx$ in the undirected sense has length $9$, so that $\lambda_{\simeq}(xxx)\geq 3$.  Since $\lambda_{\approx}(xxx)=3$~\cite{Dekking1979}, we conclude that $\lambda_{\simeq}(xxx)=3$.  

Finally, since $\lambda_{\approx}\left(x^4\right)=2$~\cite{Dekking1979}, we have $\lambda_{\simeq}\left(x^k\right)=2$ for all $k\geq 4$. 
\end{proof}

We now determine the undirected avoidability index of every binary pattern.  We consider patterns on the variables $x$ and $y$, and we say that two patterns $p$ and $q$ are \emph{equivalent} if $q\in\{p,\overline{p},\rev{p},\rev{\overline{p}}\}$, where $\overline{p}$ is the image of $p$ under the morphism defined by $x\mapsto y$ and $y\mapsto x$.  For patterns $p$ and $q$, if $p$ is equivalent to a factor of $q$, then we must have $\lambda_\simeq(q)\leq \lambda_\simeq(p)$.  We first state a lemma which gives the undirected avoidability index of some binary patterns.

\begin{Lemma}\label{AvoidabilityLemma} Let $P$ denote the set of patterns
\begin{align*}
\{xxyyx, xyxyx, xxxyxy, xxxyyx, xxxyyy, xxyxxy,\\ xxyxyy, xxyyyx, xyxxxy, xyxxyx, xyxyyx\}.
\end{align*}
We have
\begin{enumerate}[label=\textnormal{(\alph*)}]
\item $\lambda_{\simeq}(p)=2$ for every $p\in P$, and
\item $\lambda_\simeq(xyxy)=3$.
\end{enumerate}
\end{Lemma}

Before proving Lemma~\ref{AvoidabilityLemma}, we use it to determine the undirected avoidability index of every binary pattern.

\begin{theorem}
Let $p$ be a pattern on variables $x$ and $y$.  
\begin{enumerate}[label=\textnormal{(\alph*)}]
\item If $p$ is equivalent to a factor of $xyx$, then $\lambda_{\simeq}(p)=\infty$.
\item If $p$ has a factor equivalent to $xxxx$ or some pattern in $P$, then $\lambda_{\simeq}(p)=2$.
\item Otherwise, we have $\lambda_{\simeq}(p)=3$.
\end{enumerate}
\end{theorem}

\begin{proof}
For (a), if $p$ is equivalent to a factor of $xyx$, then $p$ is unavoidable in the ordinary sense, and hence unavoidable in the undirected sense by Theorem~\ref{Avoidable}.  Part (b) follows from Theorem~\ref{UnaryIndex} and Lemma~\ref{AvoidabilityLemma}(a).  Finally, for (c), since $p$ is not equivalent to a factor of $xyx$, it must be the case that $p$ has a factor equivalent to either $xx$ or $xyxy$.  By Theorem~\ref{UnaryIndex} and Lemma~\ref{AvoidabilityLemma}(b), it follows that $\lambda_{\simeq}(p)\leq 3$.  Further, since no factor of $p$ is equivalent to a member of $P$, backtracking on the alphabet $\{x,y\}$ reveals that there are only finitely many possibilities for $p$.  Every such pattern $p$ is equivalent to a factor of some pattern in Table~\ref{BacktrackingTable}.  It follows that $\lambda_{\simeq}(p)\geq 3$. 
\end{proof}

\begin{table}
\centering
\begin{tabular}{c | c}
Pattern $q$ & Length of longest binary word avoiding $q$ up to $\simeq$\\\hline
$xxxyy$ & 33\\
$xxyxy$ & 30\\
$xyxxy$ & 96\\
$xyyyx$ & 39\\
$xxxyxxx$ & 62\\
\end{tabular}
\caption{Backtracking results for some binary patterns}\label{BacktrackingTable}
\end{table}

\begin{proof}[Proof of Lemma~\ref{AvoidabilityLemma}]
For part (a), for every $p\in P$, we give a binary morphic word $\bm{w}_p$ that avoids $p$ in the undirected sense.  For part (b), we have $\lambda_{\simeq}(xyxy)\geq 3$ by Table~\ref{BacktrackingTable}, since $xyxy$ is a factor of $xxyxy$, and we give a ternary morphic word $\bm{w}_{xyxy}$ that avoids $xyxy$ in the undirected sense.

For every $p\in P\cup\{xyxy\}$, the morphism(s) used to construct $\bm{w}_p$ are given immediately after the proof.  We describe how to verify that $\bm{w}_p$ avoids $p$.  The key tool is Cassaigne's algorithm~\cite{Cassaigne1994}.

First, we verify that $\bm{w}_p$ avoids $p$ in the ordinary sense using Cassaigne's algorithm~\cite{Cassaigne1994}.  Thus, if $\bm{w}_p$ contains an undirected instance of $p$, then in this instance, at least one of the variables $x$ or $y$ must have been replaced by a nonpalindromic reversible factor of $\bm{w}_p$.  By an exhaustive search, $\bm{w}_p$ has only finitely many reversible factors.  For each nonpalindromic reversible factor $u$ of $\bm{w}_p$, we once again use Cassaigne's algorithm to verify that $\bm{w}_p$ avoids all patterns with constants obtained from $p$ by replacing every appearance of $x$ (or $y$) with either $u$ or $\rev{u}$.  Finally, for every pair of nonpalindromic reversible factors $u$ and $v$ of $\bm{w}_p$, there are finitely many factors obtained from $p$ by replacing every appearance of $x$ with either $u$ or $\rev{u}$, and every appearance of $y$ with either $v$ or $\rev{v}$.  We verify that none of these factors appear in $\bm{w}_p$.
\end{proof}

For $p\in \{xxxyyx, xxxyyy, xxyxyy, xxyyyx, xyxyyx\},$ define $\bm{w}_p=f_p^\omega(\tt{0})$, where the morphism $f_p:\Sigma_2^*\rightarrow\Sigma_2^*$ is defined as follows: 
\begin{multicols}{2}
\noindent
\begin{align*}
f_{xxxyyx}(\tt{1})&=\tt{1112}\\
f_{xxxyyx}(\tt{2})&=\tt{1222}\\[5pt]
f_{xxxyyy}(\tt{1})&=	\tt{12112}\\
f_{xxxyyy}(\tt{2})&=\tt{21221}\\[5pt]
f_{xxyxyy}(\tt{1})&=	\tt{1211221}\\
f_{xxyxyy}(\tt{2})&=\tt{2122112}
\end{align*}
\begin{align*}
f_{xxyyyx}(\tt{1})&=	\tt{112212211212}\\
f_{xxyyyx}(\tt{2})&=\tt{221121122121}\\[5pt]
f_{xyxyyx}(\tt{1})&=	\tt{11122212}\\
f_{xyxyyx}(\tt{2})&=\tt{11122212}.
\end{align*}
\end{multicols}

\noindent
For $p\in \{xxxyxy, xxyxxy, xyxxxy, xyxxyx\},$ define $\bm{w}_p=g_p(f^\omega(\tt{0}))$, where the morphism $f:\Sigma_3^*\rightarrow \Sigma_3^*$ is the well-known morphism defined by
\begin{align*}
f(\tt{1})&= \tt{123}\\
f(\tt{2})&= \tt{13}\\
f(\tt{3})&= \tt{2},
\end{align*}
and the morphism $g_p:\Sigma_3^*\rightarrow\Sigma_2^*$ is defined as follows:
\begin{multicols}{2}
\noindent
\begin{align*}
g_{xxxyxy}(\tt{1})&=\tt{11}\\
g_{xxxyxy}(\tt{2})&=\tt{12212}\\
g_{xxxyxy}(\tt{3})&=\tt{12221}\\[5pt]
g_{xxyxxy}(\tt{1})&=\tt{111}\\
g_{xxyxxy}(\tt{2})&=\tt{112}\\
g_{xxyxxy}(\tt{3})&=\tt{222}
\end{align*}
\begin{align*}
g_{xyxxxy}(\tt{1})&=\tt{1121}\\
g_{xyxxxy}(\tt{2})&=\tt{2111}\\
g_{xyxxxy}(\tt{3})&=\tt{122212}\\[5pt]
g_{xyxxyx}(\tt{1})&=\tt{12}\\
g_{xyxxyx}(\tt{2})&=\tt{11}\\
g_{xyxxyx}(\tt{3})&=\tt{222}.
\end{align*}
\end{multicols}

\noindent
For $p$ in $\{xyxy, xxyyx, xyxyx\}$, define $\bm{w}_p=g_p(f_p^\omega(\tt{0}))$, where the morphisms $f_{xyxy}:\Sigma_4^*\rightarrow\Sigma_4^*$ and $g_{xyxy}:\Sigma_4^*\rightarrow\Sigma_3^*$ are  defined by
\begin{multicols}{2}
\noindent
\begin{align*}
f_{xyxy}(\tt{1})&=\tt{1324}\\
f_{xyxy}(\tt{2})&=\tt{1423}\\
f_{xyxy}(\tt{3})&=\tt{134}\\
f_{xyxy}(\tt{4})&=\tt{143}
\end{align*}
\begin{align*}
g_{xyxy}(\tt{1})&=\tt{1223}\\
g_{xyxy}(\tt{2})&=\tt{123}\\
g_{xyxy}(\tt{3})&=\tt{11233}\\
g_{xyxy}(\tt{4})&=\tt{112233},
\end{align*}
\end{multicols}
\noindent
the morphisms $f_{xxyyx}:\Sigma_4^*\rightarrow\Sigma_4^*$ and $g_{xxyyx}:\Sigma_4^*\rightarrow\Sigma_2^*$ are defined by
{\footnotesize
\begin{align*}
f_{xxyyx}(\tt{1})&=\tt{123} & g_{xxyyx}(\tt{1})&=\tt{12211122211121221112221211222111}\\
f_{xxyyx}(\tt{2})&=\tt{214} & g_{xxyyx}(\tt{2})&=\tt{21122211122212112221112122111222}\\
f_{xxyyx}(\tt{3})&=\tt{12343} & g_{xxyyx}(\tt{3})&=\tt{122111222111212211122212112221112122111222111}\\
f_{xxyyx}(\tt{4})&=\tt{21434} \ \ \ & g_{xxyyx}(\tt{4})&=\tt{211222111222121122211121221112221211222111222},
\end{align*}}%
\noindent
and the morphisms $f_{xyxyx}:\Sigma_6^*\rightarrow\Sigma_6^*$ and $g_{xyxyx}:\Sigma_6^*\rightarrow\Sigma_2^*$ are defined by
\begin{align*}
f_{xyxyx}(\tt{1})&=\tt{1235} & 
g_{xyxyx}(\tt{1})&=\tt{122111121122221211221111}\\
f_{xyxyx}(\tt{2})&=\tt{2146} & 
g_{xyxyx}(\tt{2})&=\tt{211222212211112122112222}\\
f_{xyxyx}(\tt{3})&=\tt{12356} & 
g_{xyxyx}(\tt{3})&=\tt{1221111211222212112211112222}\\
f_{xyxyx}(\tt{4})&=\tt{21465} & 
g_{xyxyx}(\tt{4})&=\tt{2112222122111121221122221111}\\
f_{xyxyx}(\tt{5})&=\tt{1465} & 
g_{xyxyx}(\tt{5})&=\tt{122111121221122221111}\\
f_{xyxyx}(\tt{6})&=\tt{2356} & 
g_{xyxyx}(\tt{6})&=\tt{211222212112211112222}.
\end{align*}
Note that for $p$ in $\{xyxy, xxyyx, xyxyx\}$, we found the morphisms $f_p$ and $g_p$ by first finding a morphic construction of the run-length encoding $\bm{u}_p$ for the word $\bm{w}_p$.  The word $\bm{w}_p$ can be obtained from its run-length encoding $\bm{u}_p$ by a finite-state transducer.  It is known that the finite-state transduction of a morphic word is morphic~\cite{Dekking1994} (see also~\cite[Section 7.9]{AlloucheShallit2003}), and the proof is constructive, so we found $f_p$ and $g_p$ by following the construction of this proof.

\section{Conclusion}

Our confirmation of Conjecture~\ref{conj} for $k\in\{4,5,\ldots, 21\}$ leaves little doubt (at least in \emph{our} minds) that the conjecture does indeed hold for all $k\geq 4$.  However, confirming the conjecture for all $k\geq 22$ still presents a significant challenge.  In order to confirm the conjecture for infinitely many values of $k$, we will likely need to find a unified construction in order to provide a proof that eliminates the need for brute force searches.

We briefly place Conjecture~\ref{conj} in a broader context within the literature on repetitions in words.  We know that $\RT(k)=k/(k-1)$ for all $k\geq 5,$ we conjecture that $\URT(k)=(k-1)/(k-2)$ for all $k\geq 4$, and Samsonov and Shur~\cite{SamsonovShur2012} conjecture that $\ART(k)=(k-1)/(k-2)$ for all $k\geq 5$.  Let us fix $k\geq 5$.  Consider exponents belonging to the set of ``extended rationals'', which includes all rational numbers and all such numbers with a $+$, where $x^+$ covers $x$, and the inequalities $y\leq x$ and $y<x^+$ are equivalent.  Shur~\cite{Shur2011} proposes splitting all exponents greater than $\RT(k)$ into levels as follows:
\begin{center}
\begin{tabular}{c c c c}
$1$st level & $2$nd level & $3$rd level & $\dots$\\
$\left[\tfrac{k}{k-1}^+,\tfrac{k-1}{k-2}\right]$ & $\left[\tfrac{k-1}{k-2}^+,\tfrac{k-2}{k-3}\right]$ & $\left[\tfrac{k-2}{k-3}^+,\tfrac{k-3}{k-4}\right]$  & $\dots$
\end{tabular}
\end{center}
For $\alpha,\beta\in \left[\tfrac{k}{k-1}^+,\tfrac{k-3}{k-4}\right]$, Shur provides evidence that the language of $\alpha$-free $k$-ary words and the language of $\beta$-free $k$-ary words exhibit similar behaviour (e.g., with respect to growth) if $\alpha$ and $\beta$ are in the same level, and quite different behaviour otherwise; see~\cite{Shur2011,Shur2014}.  If the conjectured values of $\URT(k)$ and $\ART(k)$ are correct, then the undirected repetition threshold and the Abelian repetition threshold provide further evidence of the distinction between levels.

We close with some remarks concerning undirected pattern avoidance.  Let~$\sim$ be an equivalence relation on words.  It is clear that if the pattern $p$ is a factor of the pattern $q$, then any word avoiding $p$ up to $\sim$ must also avoid $q$ up to $\sim$, and it follows that $\lambda_\sim(q)\leq \lambda_\sim(p).$  In the special case that $\sim$ is $=$ or $\approx$, if $q$ encounters $p$ up to $\sim$, then any word avoiding $p$ up to $\sim$ must also avoid $q$ up to $\sim$.  This is not true when $\sim$ is $\simeq$.  For example, the pattern $xyxy$ encounters $xx$ up to $\simeq$, but the word $\tt{012021}$ avoids $xx$ up to $\simeq$ and encounters $xyxy$ up to $\simeq$.  

\begin{question}
If $q$ encounters $	p$ up to $\simeq$, then is $\lambda_{\simeq}(q)\leq \lambda_{\simeq}(p)$?
\end{question}

Finally, we remark that the study of $k$-avoidability of patterns up to $\simeq$ has implications for the $k$-avoidability of \emph{patterns with reversal} (see~\cite{CurrieLafrance2016,CurrieMolRampersad1,CurrieMolRampersad3} for definitions and examples).  In particular, if the pattern $p$ is $k$-avoidable up to $\simeq$, then all patterns with reversal  that are obtained by swapping any number of letters in $p$ with their mirror images are \emph{simultaneously $k$-avoidable}; that is, there is an infinite word on $k$ letters avoiding all such ``decorations'' of $p$.

\section*{Acknowledgements}

We thank the anonymous reviewers, whose comments helped to improve the article.


\providecommand{\bysame}{\leavevmode\hbox to3em{\hrulefill}\thinspace}
\providecommand{\MR}{\relax\ifhmode\unskip\space\fi MR }
\providecommand{\MRhref}[2]{%
  \href{http://www.ams.org/mathscinet-getitem?mr=#1}{#2}
}
\providecommand{\href}[2]{#2}

\end{document}